\documentclass{amsart}
\usepackage{amssymb}
\usepackage{graphicx,color}

\usepackage[normalem]{ulem}

\usepackage{color}
\usepackage{amsxtra}
\usepackage{amssymb}
\usepackage{mathtools}
\usepackage{enumerate}
\usepackage{nicefrac}
\mathtoolsset{showonlyrefs} 
\usepackage{ mathrsfs }
\usepackage{tikz}
\usetikzlibrary{positioning, calc}
\usetikzlibrary{shadings}
\usepackage{color}
\usepackage{hyperref}
\hypersetup{
  colorlinks   = true, 
  urlcolor     = blue, 
  linkcolor    = blue, 
  citecolor   = red 
}
  
\newtheorem{theorem}{Theorem}[section]
\newtheorem{lemma}{Lemma}[section]

\theoremstyle{definition}
\newtheorem{definition}{Definition}[section]

\newcommand{\R}{{\mathbb R}}

 \def\1{\raisebox{2pt}{\rm{$\chi$}}}

\usepackage{enumerate}

\def\undertilde#1{\mathord{\vtop{\ialign{##\crcr
$\hfil\displaystyle{#1}\hfil$\crcr\noalign{\kern1.5pt\nointerlineskip}
$\hfil\tilde{}\hfil$\crcr\noalign{\kern1.5pt}}}}}

\DeclareMathAlphabet{\mathpzc}{OT1}{pzc}{m}{it}

\title{
The limit as $s\nearrow 1$ of the fractional convex envelope
}

\author[B. Barrios]{Bego\~na Barrios}
	\address{Bego\~na Barrios \hfill\break\indent
		Departamento de An\'{a}lisis Matem\'{a}tico,
		Universidad de La Laguna
		\hfill \break \indent C/. Astrof\'{\i}sico Francisco S\'{a}nchez s/n, 
		38200 -- La Laguna, SPAIN}
	\email{bbarrios@ull.edu.es}
    
    \author[L. M. Del Pezzo]{Leandro M. Del Pezzo}
	\address{Leandro M. Del Pezzo \hfill\break\indent
            IESTA --Facultad de Ciencias Económicas y de Administración 	
            \hfill\break\indent
            Universidad de la República\hfill\break\indent
            Av. Gonzalo Ramírez 1926, 11200 Montevideo,\hfill\break\indent
            Departamento de Montevideo - Uruguay.}
\email{leandro.delpezzo@fcea.edu.uy}

\author[A. Quaas]{Alexander Quaas}
\address{A. Quaas
       \hfill\break\indent Departamento de Matem\'atica, Universidad T\'ecnica Federico Santa Mar\'ia 
        \hfill\break\indent Casilla V-110, Avda. Espa\~na, 1680 -- Valpara\'iso, CHILE.}
    \email{alexander.quaas@usm.cl}
    
\author[J. D. Rossi]{Julio D. Rossi}
	\address{Julio D. Rossi \hfill\break\indent
        Departamento  de Matem{\'a}tica, FCEyN,
        Universidad de Buenos Aires,
        \hfill\break\indent Pabellon I,
         Ciudad Universitaria (C1428BCW), 
        Buenos Aires, Argentina.}
\email{jrossi@dm.uba.ar}

\begin{document}

\maketitle

\begin{abstract} 
    We study the behavior of the fractional 
    convexity when the fractional parameter goes to 1. 
    For any notion of convexity, 
    the convex envelope of a datum prescribed on the boundary 
    of a domain is defined as the largest 
    possible convex function inside the domain that is below the datum on the boundary. 
	Here we prove that the fractional convex envelope inside 
	a strictly convex domain of a continuous and bounded exterior datum converges when
	$s\nearrow 1$ to the classical convex envelope of the restriction to the boundary of the exterior datum. 
\end{abstract}

\section{Introduction}
	The purpose of this paper is to study the limit when $s\nearrow 1$ 
	of fractional convex envelopes of a prescribed exterior datum.
	
	Let us first recall the usual notion of convexity and the notion of fractional convexity (introduced in 
	\cite{DelpeQR}) in the Euclidean space. Along the whole paper $\Omega\subset\R^N$ will denote an open, bounded, strictly convex
	(given two points $x,y \in \overline{\Omega}$ the segment $tx+(1-t)y$, $t\in (0,1)$
	is included in $\Omega$) and smooth ($C^2$) domain.
		
\subsection{Classical convexity} 
	A function $u\colon \Omega \to  {\R}$ is said to be convex in $\Omega$ if,
	for any two points $x,y \in \Omega$, it holds that
	\begin{equation} \label{convexo-usual}
		u(tx+(1-t)y) \leq tu(x) + (1-t) u(y), \quad t \in (0,1).
	\end{equation}
	Notice that $t \mapsto v(t) \coloneqq  tu(x) + (1-t) u(y)$ is just the solution to the equation 
	$v''=0$ in the interval $[0,1]$ that verifies $v(1) = u(x) $ and $v(0) = u(y)$ at the endpoints.
	We refer to \cite{Vel} for a general reference on convexity. 
	
	With this notion of convexity one can define the convex envelope inside $\Omega$ of a boundary
	datum $g\colon \partial \Omega \to  {\R}$ as the largest convex function in $\overline\Omega$
	that is below $g$ on the boundary, that is, 
	\begin{equation} \label{convex-envelope-usual}
		\Upsilon_g (x) \coloneqq \sup \Big\{v(x) \colon v 
		\mbox{ is convex in $\overline{\Omega}$ and verifies } v|_{\partial \Omega} \leq g \Big\}.
	\end{equation}
	
	In terms of a second order partial differential equation, a continuous function is convex 
	in $\Omega$ if and only if 
	\[
		\inf_{z\in\mathbb{S}^{N-1}} 
			\langle D^2 u(x) z, z \rangle 
		 \geq 0,\, x\in\Omega,
	\] 
	inside $\Omega$ in the viscosity sense, (see \cite{OS,Ober}). 
	Here $\mathbb{S}^{N-1}\coloneqq \{z\in \R^N \colon |z|=1\}$ denotes 
	the $(N-1)-$dimensional sphere. 
	
	Moreover, the convex envelope of $g$, a continuous datum on the boundary, in a strictly convex 
	domain turns out to be the unique solution to
	\begin{equation}
			\label{convex-envelope-usual-eq-1}	
			\left\{
			\begin{array}{ll}
			\displaystyle \inf_{z\in\mathbb{S}^{N-1}}
			\langle D^2 u(x) z, z \rangle = 0, \quad &x\in  \Omega, \\[6pt]
		\displaystyle	u(x) = g (x),  \quad &x\in  \partial \Omega.
		\end{array}
		\right.
	\end{equation}
	Therefore, we have that the operator that is associated to the convex envelope
	is just the infimum of the second directional derivatives of the function among all possible directions. 
	We refer to \cite{BlancRossi,HL1,Ober, OS}, and references therein.
	
	The equation \eqref{convex-envelope-usual-eq-1} has to be
	interpreted in the viscosity sense (see below for the precise definition) and the boundary condition is 
	attained with continuity. 
		
	\subsection{Fractional convexity} In \cite{DelpeQR}
	the authors introduce the following natural extension of convexity to the fractional setting. 
	Given $s\in(0,1),$ a function $u\colon\R^N \to  {\R}$ 
	is said to be {\it $s-$convex} in $\Omega$ if
	for any two points $x,y \in \Omega$ it holds that
	\begin{equation} \label{convexo-s}
		u(t x+(1-t)y) \leq v(t), \quad t \in (0,1),
	\end{equation}
	where $v$ is just the viscosity solution 
	to $\Delta^s_1 v =0$ (the 1-dimensional $s-$fractional laplacian)
	in the segment $[0,1]$ with $v(t) = u(t x+(1-t)y)$ outside the segment. That is, $v$ verifies
	\[
		\Delta^s_1 v(t)
		\coloneqq c(s)\int_{\R} \frac{v(t+r)-v(t )}{|r|^{1+2s}} \, dr = 0,
	\]
	for every $t \in (0,1)$ with 
	\[
		v(t)=u(tx+(1-t)y) \quad \mbox{ for }  t\not\in (0,1).
	\] 
	 
	The integral has to be understood  in the principal value 
	sense. Here the constant $c(s)$ that appears in front of the integral is given by the formula 
    \begin{equation} \label{cte}
         c(s):= \frac{2^{2s}s\Gamma  (s+1/2)}{\pi^{1/2} \Gamma  (1-s)}.
    \end{equation} 
    We just remark that any $c(s)$ such that $c(s) \sim (1-s)$ will also work for our purposes
    when we take the limit as $s\nearrow 1$, but the
    explicit form of \eqref{cte} is the one that corresponds to the 1-dimensional fractional Laplacian 
    (see for instance \cite{Hich}). 
         
	Notice that as usual for the fractional laplacian we have to impose an exterior datum. That is, we have to use values of $u$ outside $\Omega$ since 
	the involved operator is nonlocal, therefore
	$u$ has to be defined in the whole $\R^N$. 
	
	With this definition of $s-$convexity one can define the {\it $s-$convex envelope} of an exterior 
	datum $g \colon \R^N\setminus \Omega \to  {\R}$ as
	\begin{equation} \label{convex-envelope-s}
		\Upsilon^s_g (x) \coloneqq \sup 
		    \Big\{w(x) \colon w \mbox{ is $s-$convex in $\overline{\Omega}$ and verifies } 
			w|_{\R^N\setminus \Omega} \leq g \Big\}.
	\end{equation}
	This definition makes sense when the above set of functions is not empty. 
	In  particular, this is the case
	when there exists an extension of $g$ inside $\Omega$ that is $s-$convex and from 
	the results in \cite{DelpeQR} this holds when
	$g$ is continuous and bounded. Moreover, the function 
	$\Upsilon^s_g(x)$ is unique and $s-$convex.
	
	Concerning the equation that is associated with this
	notion of fractional convexity, we get that a function is $s-$convex if and only 
	if
	$$
	    \Lambda_1^s u(x):=\inf_{z\in\mathbb{S}^{N-1}} c(s)
				 \int_{\R} 
					\frac{u(x+tz)-u(x)}{|t|^{1+2s}} 
					\, dt \geq 0,\, x\in\Omega,
	$$
	 in the viscosity sense. Moreover, we know that, assuming 
	that $g$ is continuous and bounded, the $s-$convex envelope is continuous 
	in $\overline{\Omega}$ (up to the boundary) with  
	$u|_{\partial \Omega} =g|_{\partial \Omega}$ and is characterized as being the
	unique viscosity solution to
	\begin{equation} \label{convex-envelope-s-eq} 
	    \left\{
			\begin{array}{ll}
				 \Lambda_1^s u(x)= 0, \quad & x \in \Omega ,\\[6pt]
				u (x) = g (x), \quad & x \in \R^N\setminus \Omega,
			\end{array}
		\right.
    \end{equation}
    (see \cite{DelpeQR}). This equation also appears in \cite{paper} in connection with
	a proof of ABP estimates for solutions to fractional equations.  It also appears in nonlocal version of Monge–Ampere equation  introduced in \cite{caff}, see \cite{DelpeQR}.
	Close related problems involving maxima/minima among different subspaces than $\mathbb{S}^{N-1}$ 
	were considered in \cite{BdpQR,Bi2,Bi3,HolderRef,RoRu} where the truncated fractional laplacians where introduced.   
	
\subsection{Classical convexity is the limit of fractional convexity
	as $s\nearrow 1$} As we have mentioned, our main goal here
	is to show that the classical convex envelope can be obtained 
    taking the limit as $s\nearrow 1$ of the $s-$convex envelopes. That is,
	    
	\begin{theorem} \label{teo.main.intro} Let $\Omega$ be a smooth, strictly convex and bounded domain in $\R^N$. Given a continuous and bounded exterior datum $g: \R^N \setminus \Omega \mapsto \R$,
	    let $u_s:=\Upsilon^s_g$ be the sequence of $s-$convex envelopes of $g$ inside
	    $\Omega$ and $u:=\Upsilon_g$ the convex envelope of $g|_{\partial \Omega}$. Then, 
	    $\{u_s\}$ converges uniformly in $\overline{\Omega}$ to $u$ as $s\nearrow 1$.
    \end{theorem} 
	
	Notice that this result reaffirms that the previous notion of fractional convexity 
	is the natural one since it recovers 
	the classical convex envelope of $g|_{\partial \Omega}$ 
	taking the limit as $s\nearrow 1$. In fact, if one wants to approximate the classical
	convex envelope of a continuous function $\hat{g}\colon \partial \Omega \mapsto \R$ 
	by fractional convex envelopes, one just has to extend $\hat{g}$ to 
	$g\colon\R^N \setminus \Omega\mapsto \R$ in such a way that the 
	extension is continuous and bounded and then use Theorem \ref{teo.main.intro} by taking the limit as 
	$s\nearrow 1$ of the $s-$convex envelopes of the extension $g$.
		
	One of the ideas for the proof of Theorem \ref{teo.main.intro} is based on the, today well known, fact that the fractional Laplacian converges to the classical Laplacian. That is,
	$(-\Delta)^s \to -\Delta$ as $s\nearrow 1$ (see \cite{Hich}) and also \cite{Stinga}). Hence, our result extends this 
	convergence to the $s-$convex envelopes of a datum. Here we use that 
	 in an interval of the real line the solutions to 
	$\Delta^s_1 v =0$ (the 1-dimensional $s-$fractional laplacian) with a fixed exterior datum 
	$g$ converge to the solution to $v'' =0$ in the interval with boundary data $g$ on the endpoints of 
	the interval (see \cite{Hich}).  In this limit one needs the fact that the normalizing constant verifies $c(s) \sim (1-s)$. 
	
	Let us comment briefly the hypotheses on the data, $\Omega$ and $g$. The hypothesis that 
	$\Omega$ is strictly convex is used in order to show that the $s-$convex envelope and the classical 
	convex envelope are continuous up to the boundary for an exterior datum $g$ continuous and 
	bounded (see \cite{BlancRossi,DelpeQR,OS}).
	It is good to bear in mind that the definition of $s-$convexity it is needed to be assumed that        
	the exterior datum $g$ is such that we can solve the Dirichlet problem
	for the $1-$dimensional fractional $s-$laplacian in every segment inside $\Omega$ 
	(this involves values of $g$ in the line that contains this segment). We require  that the datum $g$ is 
	continuous and bounded that guarantees that there is a solution for the $1-$dimensional 
	fractional $s-$laplacian in every segment inside $\Omega$ with exterior datum $g$ that is 
	uniformly bounded above and below by bounds for $g$ (see \cite{DelpeQR}). 
		
	Our strategy to prove Theorem \ref{teo.main.intro}, and to
	show that $u_s$ converge to the usual convex envelope
	as $s\nearrow 1$, is to use the well known half relaxed limits.
	These are given by  
	$$
        u^*(x)\coloneqq\limsup_{s\nearrow 1}{}^*
                u_s(x) = \sup\left\{ \limsup_{k\to \infty, s\nearrow 1}  
                            u_s(x_k) \colon x_k \to x \right\} 
    $$ 
    and 
    $$
        u_*(x)\coloneqq \liminf_{s\nearrow 1}{}_*
                u_s(x) = \inf \left\{ \liminf_{k\to \infty, s\nearrow 1}
                u_s(x_k) \colon x_k \to x\right\}.
    $$
    We show that $u^*$ is a subsolution to \eqref{convex-envelope-usual-eq-1} and $u_*$ 
    is a supersolution. From the comparison principle for  \eqref{convex-envelope-usual-eq-1} we 
    obtain $u^* \leq u_*$  (notice that the reverse inequality trivially holds) and 
    hence we conclude that $u^*=u_*$ proving the desired convergence result.

\section{Proof of Theorem \ref{teo.main.intro}.} \label{sect.2}

	First, let us state precisely the notions of viscosity solution
	both for the local problem \eqref{convex-envelope-usual-eq-1} and for the nonlocal problem
	\eqref{convex-envelope-s-eq} that we will use.
	
\subsection{Basic notations and definitions of viscosity solutions.} 

    The
    lower semicontinuous envelope, $\underline{u}$, and the upper semicontinuous envelope, 
    $\overline{u}$, of $u$, are given by
    \[
        \underline{u}(x):=\sup_{r>0}\inf_{y\in B_r(x)} u(y)
            \quad \text{and} \quad
        \overline{u} (x):=\inf_{r>0}\sup_{y\in B_r(x)} u(y).
    \]
    
    We will use the upper and lower envelopes to define sub and supersolutions
    in the viscosity sense. 

    \subsubsection{Classical convexity} Recall form the introduction that
  the convex envelope of a continuous boundary datum $g$ is a
  viscosity solution to \eqref{convex-envelope-usual-eq-1}. Let
  us state the precise meaning of being a viscosity solution
  to \eqref{convex-envelope-usual-eq-1}.

    \begin{definition} \label{def.sol.viscosa.1}
        A function  $u:\Omega \mapsto \R$  verifies
        \begin{equation}
			\label{convex-envelope-usual-eq-1.sect}	
			\left\{
			\begin{array}{ll}
			\displaystyle \inf_{z\in\mathbb{S}^{N-1}} 
			\langle D^2 u(x) z, z \rangle  = 0, \quad &x\in  \Omega, \\[6pt]
		\displaystyle	u(x) = g (x),  \quad &x\in  \partial \Omega,
		\end{array}
		\right.
	\end{equation}
        \emph{in the viscosity sense} if
        \begin{enumerate}
                \item (viscosity supersolution) ${u} \geq g$ on $\partial \Omega$ 
                and for every $\phi\in C^{2}$ such that ${u}-\phi $ has a strict
                minimum at some point $x \in \Omega$,
                we have
                $$
                    \inf_{z\in\mathbb{S}^{N-1}} 
			                    \langle D^2 \phi (x) z, z \rangle  \leq 0.
                $$
               \item (viscosity subsolution) ${u} \leq g$ 
               on $\partial \Omega$ and for every 
               $\psi \in C^{2}$ such that $ {u}-\psi $ has a
                strict maximum at some point $ x \in {\Omega}$, we have
                $$
                    \inf_{z\in\mathbb{S}^{N-1}}
			        \langle D^2 \psi (x) z, z \rangle \geq 0.
             $$
        \end{enumerate}
    \end{definition}
        
        With the previous definition we have the following theorem from 
        \cite{HL1} that gives existence and uniqueness for 
        \eqref{convex-envelope-usual-eq-1}.

    \begin{theorem}[See \cite{HL1}] 
        Let $\Omega$ be a smooth, strictly convex and bounded domain in $\R^N$. 
        Then, for every $g\in C(\partial \Omega)$,  the problem
        \[
            \left\{
            \begin{array}{ll}
                \displaystyle 
                \inf_{z\in\mathbb{S}^{N-1}} 
			                \langle D^2 u (x) z, z \rangle  = 0 , \quad & \mbox{ in } \Omega, \\[5pt]
                u=g , \quad & \mbox{ on } \partial \Omega,
            \end{array}
            \right.
        \]
        has a unique viscosity solution $u\in C(\overline{\Omega})$.

        Moreover, a comparison principle holds. That is, a viscosity supersolution $\overline{u}$ and a viscosity subsolution $\underline{u}$
        are ordered, $\underline{u} \leq \overline{u}$ 
        inside $\Omega$. 
    \end{theorem}

\subsubsection{Fractional convexity}
	
	We use the notion of viscosity solution from \cite{BarChasImb}, which is the nonlocal extension of 
	the classical theory, see \cite{CIL}, in order to give the proper definition of viscosity solutions of \eqref{convex-envelope-s-eq}. Moreover, to state the precise definition of solution, we need the following: 
	Given $g\colon\R^N\setminus\Omega \to \R,$  for a function 
	$u \colon \overline{\Omega} \to \R$
	we define the upper $g$-extension of $u$ as
	\begin{equation*}
		u^g(x) \coloneqq 
			\left \{ 
				\begin{array}{ll} 
					u(x), \quad & \mbox{if} \ x \in \Omega, \\[6pt] 
					g(x), \quad & \mbox{if} \ x \in \R^N\setminus\overline{\Omega},\\[6pt]
					\max \{ u(x), g(x) \}, \quad & \mbox{if} \ x \in \partial \Omega.
				\end{array} 
			\right . 
	\end{equation*}
	In the analogous way we define $u_g$, the lower $g$-extension of $u$, replacing $\max$ by $\min$. 
	
	An important fact, that can be easily 
	verified, is that for 
	any continuous function $g \colon \R^N\setminus\Omega \to \R$ and
	any upper semicontinuous function $u \colon \overline{\Omega} \to \R,$  it holds that
	$u^g$ is the upper semicontinuous envelope of  
	$\mathbf{1}_{\overline{\Omega}} + g \mathbf{1}_{\R^N\setminus\overline{\Omega}}.$

	We now introduce a useful notation, for $\delta > 0$ we write
		\[	
		E_{z, \delta}(u, \phi, x)\coloneqq I^1_{z, \delta}(\phi, x)+I^2_{z, \delta}(u, x),
	\] 
	with
	\[
		\begin{array}{l}
			\displaystyle 
				I^1_{z, \delta}(\phi, x)
				\coloneqq c(s)\int_{-\delta}^\delta \frac{\phi (x+t z)-\phi(x)}{|t|^{1+2s}}dt,\\[15pt]
			\displaystyle  
				I^2_{z, \delta}(u, x)\coloneqq c(s)\int_{\R \setminus (-\delta,\delta)} 
					\frac{u^g( x+t z)-u( x)}{|t|^{1+2s}}dt,
		\end{array}
	\]
	and then define 
	\begin{equation*}
		E_\delta(u, \phi, x) \coloneqq 
		- \inf_{z\in \mathbb{S}^{N-1}}
			E_{z, \delta}(u, \phi, x).
	\end{equation*}

	Now we can introduce our notion of viscosity solution testing with $N-$dimensional functions as usual. 

	\begin{definition}\label{defsol.44}
		A bounded upper semicontinuous function $u \colon \R^N \to \R$  
		is a viscosity solution to the Dirichlet problem \eqref{convex-envelope-s-eq} 
		if it verifies
		
		\begin{enumerate}
            \item (viscosity supersolution)
		        ${u} \geq g$ in $\R^N\setminus\overline{\Omega}$ and if 
		        for each $\delta > 0$ and $\phi \in C^2(\R^N)$ such that $x_0$ is a minimum 
		        point of $u- \phi$ in $B_\delta(x_0)$, then
		        \begin{equation*}
			        \begin{array}{ll}
				        \displaystyle
					        E_\delta(u^g, \phi, x_0) \geq 0, & \quad \mbox{if} \ x_0 \in \Omega, \\[6pt]
				        \displaystyle 
					        \min \left\{ E_\delta(u^g, \phi, x_0), u(x_0) - g(x_0) \right\} \geq  0, 
					        & \quad \mbox{if} \ x_0 \in \partial \Omega.
			        \end{array}
		        \end{equation*}
 		    \item (viscosity subsolution)
		        $u \leq g$ in $\R^N\setminus\overline{\Omega}$ and if 
		        for each $\delta > 0$ and $\phi \in C^2(\R^N)$ such that $x_0$ is a maximum 
		        point of $u - \phi$ in $B_\delta(x_0)$, then
		        \begin{equation*}
			        \begin{array}{ll}
				        \displaystyle
					        E_\delta(u^g, \phi, x_0) \leq 0, & \quad \mbox{if} \ x_0 \in \Omega, \\[6pt]
				        \displaystyle 
					        \min \left\{ E_\delta(u^g, \phi, x_0), u(x_0) - g(x_0) \right\} \leq  0, 
					        & \quad \mbox{if} \ x_0 \in \partial \Omega.
			        \end{array}
		        \end{equation*}
		\end{enumerate}
    \end{definition}

	As in the local case, we also have existence and uniqueness for viscosity
    solutions to \eqref{convex-envelope-s-eq} as the following result states.

    \begin{theorem}[See \cite{DelpeQR}] 
        Let $\Omega$ be a smooth, strictly convex and bounded domain in $\R^N$. 
        Then, for every continuous and bounded $g\colon\R^N \setminus \Omega \mapsto \R$, 
        the problem
        $$
            \left\{
                \begin{array}{ll}
                    \displaystyle 
                    \inf_{z\in\mathbb{S}^{N-1}}c(s) \int_{\R} 
						        \frac{u(x+tz)-u(x)}{|t|^{1+2s}} 
						        \, dt  = 0,  \quad & \mbox{ in } \Omega, \\[5pt]
                    u=g,  \quad & \mbox{ in } \R^N \setminus \Omega,
                \end{array}
            \right.
        $$
        has a unique viscosity solution $u\in C(\overline{\Omega})$ with $u=g$ on $\partial \Omega$.

        Moreover, a comparison principle holds.
        A viscosity supersolution $\overline{u}$ and a viscosity subsolution $\underline{u}$
        are ordered, $\underline{u} \leq \overline{u}$ 
        inside $\Omega$. 
    \end{theorem}

\subsection{The limit as $s\nearrow 1$.} 

	We first show a key lemma that controls (uniformly in $s$ for $s$
	close to 1) from above
	any $s-$convex function close to a boundary point.
	
	\begin{lemma} \label{lema.clave.borde}
	    Assume that $u_s$ is $s-$convex in $\Omega$ with an exterior datum
	    continuous and bounded $g$. Then, given 
	    $x_0\in \partial \Omega$ and $\eta>0$ there exists 
	    $\delta >0$ (uniform in $s$ for $s$ close to 1) 
	    such that
        $$
           u_s(x) \leq g(x_0) + \eta, \quad \mbox{ for }|x-x_0| < \delta.
        $$ 
     \end{lemma}

    \begin{proof} 
        From \cite{DelpeQR} we know that $u$, an $s-$convex function, 
        is a viscosity subsolution to 
        $$
            \inf_{z\in\mathbb{S}^{N-1}} c(s) \int_{\R} 
					\frac{u(x+tz)-u(x)}{|t|^{1+2s}} 
					\, dt = 0.
		$$
		Hence, by using \cite[Section 3]{DelpeQR}, we get that, for every $x\in \Omega$ and every direction $z\in \mathbb{S}^{N-1}$ 
		we get that
		$$
			c(s) \int_{\R} \frac{u(x+tz)-u(x)}{|t|^{1+2s}} \, dt \geq 0,
		$$
		in the viscosity sense.
		Then, $t \mapsto u(x+tz)$ is a subsolution to the $1-$dimensional fractional Laplacian
		as long as $x+tz \in \Omega$ (that is, in some interval that contains the origin). 
		With this in mind we aim to make a comparison argument in order to obtain an upper bound.

        Then, given $x_0\in\partial\Omega,$ for $\hat{x}\in \Omega$ we consider the line 
        $$
            x_0 + t \frac{\hat{x}-x_0}{|\hat{x}-x_0|}, \quad t \in \R.
        $$
        Notice that, since $\Omega$ is strictly convex we have that 
        $$
            x_0 + t \frac{\hat{x}-x_0}{|\hat{x}-x_0|} \in \Omega,
            \quad \mbox{ for }t\in (0,|x_0-\hat{x}|). 
        $$ 

        On the other side, since $g$ is continuous, given $\eta>0$ there exists a ball $B_r (x_0)$ such that
        $$
            g(y) \leq g(x_0) + \frac{\eta}{3}, \quad y \in \left(\R^N \setminus \Omega\right)
            \cap B_r (x_0).
        $$
        For this radius $r$ there is $\theta>0$ such that
        $$
            x_0 + t \frac{\hat{x}-x_0}{|\hat{x}-x_0|} \in\left(\R^N \setminus \Omega\right)
            \cap B_r (x_0).
        $$
        for $t\in (-\theta, 0)$.

        We take now $w_s(t)$ the solution to 
        \begin{equation} \label{convex-44} 
		        \begin{cases}
			        \displaystyle c(s) \int_{\R} 
						        \frac{w(r+t)-w(t)}{|r|^{1+2s}} 
						        \, dr = 0, \quad  & t \in (0,|x_0-\hat{x}|), \\[6pt]
			        \displaystyle 		w (t) = M:=\max g, & t  \in [|x_0-\hat{x}|, +\infty) \cap
					        (-\infty, -\theta],
					        \\[6pt]
			        \displaystyle w(t) = g(x_0 ) + \frac{\eta}{3}, & 
			        t  \in (-\theta, 0].
		        \end{cases}
	        \end{equation}
	    Notice that, by \cite[Theorem  2.2]{DelpeQR}, $u_s \leq M$ in $\Omega$ and, therefore
	    we have that the exterior condition for $w$ is bigger or equal
	    than the values of $u_s(x_0 + t \frac{\hat{x}-x_0}{|\hat{x}-x_0|})$
	    for $t\not\in  (0,|x_0-\hat{x}|)$. Hence, 
	    using the comparison principle for the $1-$dimensional fractional Laplacian
	    we obtain
	    \begin{equation}\label{mabel}
	        u_s \left(x_0 + t \frac{\hat{x}-x_0}{|\hat{x}-x_0|}\right) 
	        \leq w_s(t), \quad t \in (0,|x_0-\hat{x}|).
	   \end{equation}
		
	    We observe now that $w_s$ converges uniformly as $s \nearrow 1$
	    (see \cite{Hich}) to the solution of 
	    $w''(t) =0$ in $(0,|x_0-\hat{x}|)$ with 
	    $w(0)= g(x_0 ) + \frac{\eta}{3}$
	    and $w (|x_0-\hat{x}|) =M$, that is given by
	    $$
	        w(t) = g(x_0 ) + \frac{\eta}{3} + t \frac{M-g(x_0 ) 
	        + \frac{\eta}{3}}{|x_0-\hat{x}|}.
	    $$
	    Therefore, for $t$ small (let says $t\in (0,\delta)$ for some $\delta>0$) we get
	    $$
	        w(t) \leq  g(x_0 ) + \frac{2\eta}{3}.
	    $$
	    From the uniform convergence of $w_s$ to $w$ we obtain 
		$$
	        w_s(t) \leq  g(x_0 ) + \eta.
	    $$
	    for every $s\in (s_0,1)$ and every $t\in (0,\delta)$ where $\delta$ does not depend on
	    $s$.

	    Hence, from \eqref{mabel}, we conclude that
	    $$
	        u_s \Big(x_0 + t \frac{\hat{x}-x_0}{|\hat{x}-x_0|}\Big)
	        \leq g(x_0 ) + \eta, 
	    $$
	    for every $s\in (s_0,1)$ and every $t\in (0,\kappa)$.
	    The proof is finished. 
    \end{proof}

    Next, we prove a reverse kind of inequality for the $s-$convex envelope.

    \begin{lemma} \label{lema.clave.borde.2}
        Assume that $u_s$ is the solution to \eqref{convex-envelope-s-eq} 
        in $\Omega$ with an exterior datum continuous and bounded $g$. 
        Then, given $x_0\in \partial \Omega$ and $\eta>0$ 
        there exists $\delta >0$ (uniform in $s$ for $s$ close to 1) such that
        $$
            u_s (x) \geq g(x_0) - \eta, \quad \mbox{ for }|x-x_0| < \delta,
        $$ 
    \end{lemma}

    \begin{proof} 
        To simplify the notation, assume that $x_0 =0$ and that 
        $\Omega \subset \{(x_1,\dots,x_N)\colon x_1 >0\}$ 
        (here we are using that $\Omega$ is strictly convex). 
        Let 
        $$
            u_1 (x_1,\dots,x_N):= g(0)- \frac{\eta}{2} - K x_1.
        $$
        Given $\eta>0$ we can choose $K$ large and $\kappa$ small enough such that
        \begin{equation}\label{bar}
            u_1 (x) < g(x), \quad \mbox{ in }
            \{ x\in \R^N \setminus \Omega \colon d(x,\partial \Omega) \leq \kappa\} 
            \coloneqq  (\R^N \setminus \Omega)_\kappa.
        \end{equation}
        Notice that here we are using that $\Omega$ is strictly convex. We also
        remark that $(\R^N \setminus \Omega)_\kappa$ is a narrow strip around
        $\partial \Omega$ outside $\Omega$. Let us now define
        $$u_2 (x):= u_1 (x)+ \varepsilon |x|^2,\quad x\in\mathbb{R}^{N},$$
        where $\varepsilon>0$ in orden to guarantee that, by \eqref{bar}, 
       \begin{equation}\label{iren}
            u_2 (x)< g(x), \quad x\in (\R^N \setminus \Omega)_\kappa.
        \end{equation}
        Then, the function $u_2$ is a $C^2$ strictly convex
        (it satisfies $\langle D^2 u (x) z,  z \rangle \geq \varepsilon >0 $
        for every $x\in \Omega \cup (\R^N \setminus \Omega)_\kappa$ and every $z \in \mathbb{S}^{N-1}$. 

        Finally, we define 
        $$
            \widetilde{u} (x):= \left\{
                \begin{array}{ll}
                    \displaystyle u_2 (x), \quad & 
                   x \in \Omega \cup (\R^N \setminus \Omega)_\kappa, \\[6pt]
                    \displaystyle \min g, \quad & 
                    x \in \R^N \setminus(\Omega \cup (\R^N \setminus \Omega)_\kappa).
                \end{array} \right.
        $$
     that clearly satisfies
        \begin{equation} \label{cond.borde}
            \widetilde{u}< g(x), \quad  x\in \R^N \setminus \Omega.
        \end{equation}

        We claim that $\widetilde{u}$ is a $s-$convex function for every
        $s$ close to 1. For that, thanks to the results
        in \cite{DelpeQR}, we have to check that 
        $$
            c(s)  \int_{\R} 
					        \frac{\widetilde{u} (x+tz) - \widetilde{u} (x)}{|t|^{1+2s}} 
					        \, dt \geq 0,
        $$
	    for every $x \in \Omega$ and every $z \in \mathbb{S}^{N-1}$. Notice that, since $\widetilde{u}$
        is $C^2 (\Omega)$ we can check the previous inequality pointwise
        without using the viscosity theory.  That is, for $x \in \Omega$ and $z \in \mathbb{S}^{N-1}$ we write the involved singular integral as
        \[ 
            E^s_{z,\delta} (\widetilde{u}, x)
            = I^{1,s}_{z, \delta}(\widetilde{u}, x)+I^{2,s}_{z, \delta}(\widetilde{u}, x),
        \]
        with
        \[
            \displaystyle 
                I^{1,s}_{z, \delta}(\widetilde{u}, x)
                = c(s) \int_{-\delta}^\delta \frac{\widetilde{u} (x+t z)
                -\widetilde{u}(x)}{|t|^{1+2s}}dt,
        \]
        and
        \[    
        \displaystyle  
                I^{2,s}_{z, \delta}(\widetilde{u}, x) = c(s) \int_{\R \setminus (-\delta,\delta)} 
                    \frac{\widetilde{u}( x+t z)-\widetilde{u}(x)}{|t|^{1+2s}}dt.
        \]
        Here we choose $\delta <\kappa$ so that the integral $I^{1,s}_{z, \delta}(\widetilde{u}, x)
        $ involves points in $\Omega \cup (\R^N \setminus \Omega)_\kappa$ when $x \in \Omega$.

        Since $\widetilde{u}$ is bounded, one can check that
        \[
            |I^{2,s}_{z, \delta}(\widetilde{u}, x)| \leq 2\|\widetilde{u}\|_\infty c(s) 
            \int_{\R \setminus (-\delta,\delta)} 
            \frac{1}{|t|^{1+2s}}dt = \frac{2\|\widetilde{u}\|_\infty c(s) }{s\delta^{2s}},
        \]
        which, by using the fact that $c(s) \sim (1-s)$ as $s \nearrow 1$, implies that,
        \[
            |I^{2,s}_{z, \delta}(\widetilde{u}, x)| \to 0, \quad \text{as} \quad s\to 1,
        \]
        for a fixed $\delta>0$. Notice that this limit is uniform for $x\in \Omega$ and for $z \in \mathbb{S}^{N-1}$.  

        For the first integral, using that $\widetilde{u}(x+tz)$ is a second order
        polynomial around $t=0$, we get 
        \[
            I^{1,s}_{z, \delta}(\widetilde{u}, x)
            = \frac{c(s) \delta^{2-2s}}{2(1-s)} \langle D^2\widetilde{u}(x) z, z \rangle . 
        \]
        Hence, using the precise expression for $c(s)$ (that implies $c(s) \sim 2(1-s)$ and hence $\frac{c(s) \delta^{2-2s}}{2(1-s)}\to 1$ as 
        $s \nearrow 1$),
        we obtain
        \[
            \lim_{\delta \to 0^+}\lim_{s\to 1^-} I^{1,s}_{z, \delta}(\widetilde{u}, x) = 
            \langle D^2\widetilde{u} (x) z, z \rangle, 
        \]
        uniformly in $x\in \Omega$ and in $z \in \mathbb{S}^{N-1}$.

        Therefore, collecting the previous results, we obtain that
         it holds that
         \[
            \lim_{\delta \to 0^+}\lim_{s\to 1^-}E^s_{z,\delta} (\widetilde{u}, x) = \langle D^2
             \widetilde{u}(x) z,  z \rangle \geq \varepsilon >0. 
         \]
         Therefore, we get that 
          \[ 
            E^s_{z,\delta} (\widetilde{u}, x) >0,
         \]
         for every $x\in \Omega$ and every $z \in \mathbb{S}^{N-1}$
         and every $s$ close to $1$.
         Hence $\widetilde{u}$ is $s-$convex in $\Omega$ for every $s$ close to $1$.

        Using the definition of the $s-$convex envelope, we get 
        $$
            u_s (x) \geq \widetilde{u} (x), \quad x \in \R^N,
        $$
        and since by \eqref{iren} we have that 
        $$
            \widetilde{u} (x)
            < g(x), \quad \mbox{ for } |x|
            < \kappa,
        $$
        we obtain that there exists $\delta$ small enough ($\delta\leq\kappa$) such that 
        $$
            u_s (x) \geq \widetilde{u} (x) = g(0)- \eta, \quad \mbox{ for } |x|
            < \delta,
        $$
        as we wanted to prove.
    \end{proof}
 
    Now we review the half-relaxed limits of a sequence of functions. 

    \begin{definition} \label{half-relax}
        Let $u_s$ be a bounded sequence of functions. We define the half relaxed limits as follows
        \begin{equation} \label{limsup}
            u^*(x) \coloneqq  \limsup_{s\nearrow 1}{}^*
            u_s(x) = \sup\left\{ \limsup_{k\to \infty, s\nearrow 1}
            u_s(x_k) \colon x_k \to x \right\}
        \end{equation}
        and
        \begin{equation} \label{liminf.99}
            u_*(x) \coloneqq  \liminf_{s\nearrow 1}{}_*
            u_s(x) = \inf \left\{ \liminf_{k\to \infty, s\nearrow 1}
            u_s(x_k) \colon x_k \to x\right\}.
        \end{equation}
    \end{definition}

    Notice that the half relaxed limit $u^*$ is an upper semicontinuous function and the half relaxed limit
    $u_*$ is a lower semicontinuous function. Moreover, we always have
    $$
        u^* \geq u_*.
     $$

    Next, we show some usual properties of the half relaxed limits that are 
    useful to take limits in the viscosity sense. These properties are standard 
    but we include the proof for completeness. 
    

    \begin{lemma} \label{lema.propiedades} Let $\{u_s\}$ a sequence of upper semicontinuous functions in $\Omega$ and 
        $$
            u:= u^*=\limsup_{s\nearrow 1}{}^* u_s.
        $$
        Let be $\phi \in C^2$. If that $u-\phi$ has a strict local maximum at the point $x_0$,
        then there is a sequence of indexes $s_j$ and points $x_j$ such that
        \begin{enumerate}
            \item $u_{s_j}-\phi$ has a local maximum at $x_j$.
            \item $x_j$ converges to $x_0$.
        \end{enumerate}
        An analogous statement holds for 
        $$
            u:=u_* = \liminf_{s\nearrow 1}{}_* u_s,
        $$
        when $u-\phi$ has a strict local minimum at the point $x_0$.
    \end{lemma}

    \begin{proof}
        Since $u{-\phi}$ has a strict local maximum at $x_0$, we have that 
        for some $r > 0$, $u(y){-\phi(y)} < u(x_0){-\phi(x_0)}$
        for every $y \in B_r(x_0) \setminus \{x_0\}$.
        Let $\rho > 0$ be an arbitrarily small radius. We have that
        \begin{equation} \label{pepe}
            \max_{y\in B_r(x_0)\setminus B_\rho (x_0)}
                u(y) {-\phi(y)}= u(x_0){-\phi(x_0)}-\delta,
        \end{equation}
        for some $\delta > 0$.
        Thus, for $s<1$ and $s$ close to 1, 
        we have 
        \[
            u_{s} (y){-\phi (y)} \leq u(x_0){-\phi(x_0)} - \delta/2,
        \]
        for $y\in B_r(x_0)\setminus B_\rho (x_0)$ (otherwise, we contradict \eqref{pepe} using the definition half relaxed limit of $U$). Now again using the definition of the half relaxed limit, 
        we obtain that there exists a sequence $s_j$, and points $y_j \to x_0$ such that 
        $u_{s_j} (y_j)\to u(x_0)$.
        Let $x_j$ be the point where $u_{s_j}-\phi$
        achieves its maximum in $\overline{B_r}(x_0)$. This maximum cannot be smaller
        than $u_{s_j}(y_j)-\phi (y_j)$, that converges to $u(x_0)-\phi (x_0)$. Therefore, 
        for large enough $j$, we have that $x_j \in B_\rho(x_0)$. 
        Since $\rho > 0$
        is arbitrary, we conclude that $$x_j \to x_0,\quad j \to \infty.$$ Moreover,
        $$\max _{\overline{B_{r}}(x_0)}u_{s_j}(x)-\phi (x)=\max _{\overline{B_{\widetilde{r}}}(x_j)}u_{s_j}(x)-\phi (x)=u_{s_j}(x_j)-\phi (x_j),$$
        for some $\widetilde{r}>0.$
    \end{proof}

    Now we are ready to show that the upper half relaxed
    limit is a subsolution to the limit problem. 

    \begin{lemma} \label{lema.subsol}
        Let $u_s$ be a sequence of viscosity subsolutions to \eqref{convex-envelope-s-eq}. 
        Then, the half relaxed
        limit
        $$
        u^* = \limsup_{s\nearrow 1}{}^* u_s,
        $$
        is a viscosity subsolution to \eqref{convex-envelope-usual-eq-1}.
    \end{lemma}

\begin{proof}
     {
        Let $\phi$ be a $C^2$ function such that $u^*-\phi$ 
        has a strict local maximum at the point $x_0 \in \Omega$.
        Then there exists a $r>0$ such that $u^*-\phi$ is the maximum
        of  $u-\phi$ in $B_r(x_0)\subset\subset\Omega$.
        From Lemma \ref{lema.propiedades}
        we obtain that there exist $s_j \nearrow 1$ and points 
        $x_j \to x_0$  such that {$x_j$ is the maximum of $u_{s_j}-\phi$ in
        $B_r(x_j)$}. From the fact that $u_{s_j}$ is a viscosity subsolution to 
        $$
                \inf_{z\in\mathbb{S}^{N-1}}  c(s_j) \int_{\R} 
					    \frac{u(x+tz)-u(x)}{|t|^{1+2s_j}} 
					    \, dt = 0,
	    $$
	    we get that 
        $$
                        \inf_{z\in\mathbb{S}^{N-1}}
                        c(s_j)\left(  \int_{-\delta}^\delta 
					        \frac{\phi (x_j+tz) - \phi(x_j)}{|t|^{1+2s_j}} 
					    \, dt +   \int_{\R\setminus(-\delta,\delta)}  \!\!
					        \frac{u_{s_j} (x_j+tz) - u_{s_j}(x_j)}{|t|^{1+2s_j}} 
					    \, dt \right)\geq 0,
	    $$
	    for some $\delta>0$ independent of $j$.
	    Therefore, for every direction $z \in \mathbb{S}^{N-1}$  we have that,
	    \begin{equation} \label{pp}
                c(s_j)\left(   \int_{-\delta}^\delta 
					        \frac{\phi (x_j+tz) - \phi(x_j)}{|t|^{1+2s_j}} 
					    \, dt + \int_{\R\setminus(-\delta,\delta)} \!\!
					        \frac{u_{s_j} (x_j+tz) - u_{s_j}(x_j)}{|t|^{1+2s_j}} 
					    \, dt\right) \geq 0.	
	    \end{equation}
	    Now, using that $\phi \in C^2$ and $\{u_{s_j}\}$ is bounded,
    passing to the limit as $j\to\infty$ 
    (and proceeding in a similar way to the proof of Lemma \ref{lema.clave.borde.2})
    we obtain
    $$
        \langle D^2 \phi (x_0) z, z \rangle \geq 0,
    $$
    for every $z \in \mathbb{S}^{N-1}$. Thus,
    \begin{equation}\label{ecu}
            \inf_{z\in\mathbb{S}^{N-1}}
			\langle D^2 \phi(x_0) z, z \rangle \geq  0.
\end{equation}
}

To prove an inequality involving $g$ on the boundary we use Lemma \ref{lema.clave.borde}. That is,
given $x_0\in \partial \Omega$ and $\eta>0$ we get that there exists $\delta >0$ (uniform in $s$ for $s$ close to 1) such that
$$
u_s (x) \leq g(x_0) + \eta, \quad \mbox{ for }|x-x_0| < \delta.
$$ 
Thus
$$
\limsup_{s\nearrow 1}{}^* u_s (x) = u^* (x) \leq g(x) +\eta, \quad x \in \partial \Omega.
$$
Since $\eta$ is arbitrary, this shows that
$$
\limsup_{s\nearrow 1}{}^* u_s (x) = u^* (x) \leq g(x), \quad x \in \partial \Omega.
$$
Thus by \eqref{ecu} we conclude that $u^*$ is a viscosity subsolution to 
\eqref{convex-envelope-usual-eq-1}.
    \end{proof}

Analogously, we also have that the lower half relaxed
limit is a supersolution to the limit problem.

\begin{lemma} \label{lema.supersol}
Let $u_s$ be a sequence of viscosity supersolutions to \eqref{convex-envelope-s-eq}. Then, the half relaxed
limit
$$
u_* = \liminf_{s\nearrow 1}{}_* u_s,
$$
is a viscosity supersolution to \eqref{convex-envelope-usual-eq-1}.
\end{lemma}

\begin{proof}
    {
        The proof is similar to the one of Lemma \ref{lema.subsol} with a small modification 
        since here we will obtain an inequality for some $z \in \mathbb{S}^{N-1}$
        and we need to consider a sequence of directions $z_j$
        in the singular integrals when passing to the limit.
        For that, let $\phi$ be a $C^2$ function such that $u_{*}-\phi$ has a strict local minimum at the point 
        $x_0 \in \Omega$. As before, from Lemma \ref{lema.propiedades} 
        we obtain that there exist $s_j \nearrow 1$ and points $x_j \to x_0$  such that 
        $u_{s_j}-\phi$ achieves a local minimum at $x_j$. More specifically, $x_j$ is the minimum of 
        {$u_{s_j}-\phi$ in $B_r(x_j)\subset\subset\Omega$} for some $r>0$ independent of $j.$

        From the fact that $u_{s_j}$ is a viscosity supersolution to 
        \[
            \inf_{z\in\mathbb{S}^{N-1}} c(s_j) \int_{\R} 
					    \frac{u(x+tz)-u(x)}{|t|^{1+2{s_j}}} 
					    \, dt = 0,
	    \]
        we get that 
        \[
            \inf_{z\in\mathbb{S}^{N-1}} c(s_j)\left( \int_{-\delta}^\delta 
					    \frac{\phi (x_j+tz) - \phi(x_j)}{|t|^{1+2s_j}} 
					    \, dt +\int_{\R\setminus(-\delta,\delta)}\!\!
					    \frac{u_{s_j} (x_j+tz) - u_{s_j}(x_j)}{|t|^{1+2s_j}} 
					    \, dt  \right)\leq 0.
	    \]
	    Therefore, we have a sequence of directions $z_j \in \mathbb{S}^{N-1}$ 
	    (extracting a subsequence if needed we can assume that $z_j \to z$) such that 
	    \begin{equation} \label{eq.88}
                c(s_j)\left( \int_{-\delta}^\delta 
					    \frac{\phi (x_j+tz_j) - \phi(x_j)}{|t|^{1+2s_j}} 
					    \, dt + \int_{\R\setminus(-\delta,\delta)}\!\!
					    \frac{u_{s_j} (x_j+tz_j) - u_{s_j}(x_j)}{|t|^{1+2s_j}} 
					    \, dt  \right)
				     \leq \frac1j.
	    \end{equation}
				    
	    In order to pass to the limit as $j\to \infty$ we need to take care of the fact that the directions
        at which we are computing the $1-$dimensional fractional derivative also depend on $j$.
        
        Using that $\{u_{s_j}\}$ is bounded, as in the proof of Lemma \ref{lema.clave.borde.2} one can 
        check that
        \[
            c(s_j)\int_{\R\setminus(-\delta,\delta)}\!\!
					\frac{u_{s_j} (x_j+tz_j) - u_{s_j}(x_j)}{|t|^{1+2s_j}} \to 0, 
					\quad \text{as} \quad j\to\infty.
		\]
        
        For the first integral, using a second order Taylor expansion of $\phi(x_{x_j}+tz_{j})$ 
        around $t=0$, we get 
        \begin{align*}
	         c(s_j)\int_{-\delta}^\delta 
					    \frac{\phi (x_j+tz_j) - \phi(x_j)}{|t|^{1+2s_j}} 
					    \, dt &=
					    \frac{c(s_j) \delta^{2-2s_j}}{2(1-s_j)} 
					    \left( \langle D^2\phi(x_{s_j}) z_{j}, z_{j} \rangle + o_\delta(1)\right),
        \end{align*}
        Hence, using the precise expression for $c(s_j)$ (that implies $c(s_j) \sim (1-s_j)$ 
        as $j \to \infty$), recalling that that $x_j\to x_0$ and $z_{j} \to z$, and using continuity of 
        $\langle D^2\phi(x) z, z \rangle$ in both $x$ and $z$, we obtain
        \[
           \lim_{j\to\infty} c(s_j)\int_{-\delta}^\delta 
					    \frac{\phi (x_j+tz_j) - \phi(x_j)}{|t|^{1+2s_j}} 
					    \, dt= \langle D^2\phi(x_0) z, z \rangle.		    
         \]
        Therefore, collecting the previous results, we obtain that
        for any convergent sequence $z_{j} \to z$, it holds that
        \begin{align*}
	        \langle D^2\phi(x_0) z, z \rangle=& \lim_{j\to\infty}
	         c(s_j)\int_{-\delta}^\delta 
					    \frac{\phi (x_j+tz_j) - \phi(x_j)}{|t|^{1+2s_j}} 
					    \, dt\\
					    &\quad +c(s_j)\int_{\R\setminus(-\delta,\delta)}\!\!
					\frac{u_{s_j} (x_j+tz_j) - u_{s_j}(x_j)}{|t|^{1+2s_j}} \, dt\\
					&\le0.        \end{align*}
		Hence, we have that 
\begin{equation}\label{laotra}
            \inf_{z\in\mathbb{S}^{N-1}} 
			\langle D^2 \phi(x_0) z, z \rangle \leq  0.
	\end{equation}
       }
        
        To prove an inequality involving $g$ on the boundary we use, one more time, the Lemma \ref{lema.clave.borde.2} 
       to get that, given $x_0\in \partial \Omega$ and $\eta>0$, there exists $\delta >0$   
        (uniform in $s$ for $s$ close to 1) such that
        $$
            u_s (x) \geq g(x_0) - \eta, \quad \mbox{ for }|x-x_0| < \delta.
        $$ 
Thus, 
        $$
            \liminf_{s\nearrow 1}{}_* u_s (x) = u_* (x) \geq g(x) - \eta, \quad x \in \partial \Omega.
        $$
        Since $\eta$ is arbitrary, this shows that
        $$
            u_* (x) \geq g(x), \quad x \in \partial \Omega.
        $$
        Thus, by \eqref{laotra} we conclude that $u_*$ is a viscosity supersolution to \eqref{convex-envelope-usual-eq-1}.
    \end{proof}

    Now we are ready to prove our main result.

    \begin{proof}[Proof of Theorem \ref{teo.main.intro}]
        From Lemmas \ref{lema.subsol} and  \ref{lema.supersol}, we get that 
        $u^*$ and $u_*$ are viscosity sub and supersolution to \eqref{convex-envelope-usual-eq-1},
        respectively. From the comparison principle for \eqref{convex-envelope-usual-eq-1}
        (see \cite{DelpeQR}) we obtain that
        $$
            u^* (x)\leq u_* (x), \qquad x \in \overline{\Omega},
        $$
        and, since the half relaxed limits always verify the reverse inequality, 
        $$
            u^* (x)\geq u_* (x), \qquad x \in \overline{\Omega},
        $$ 
        we conclude that
        $$
            u^* (x) = u_* (x) , \qquad x \in \overline{\Omega}.
        $$
        Hence there exists the uniform limit
        $$
            u = \lim_{s\nearrow 1} u_s.
        $$

        This limit $u$ is continuous since $u=u^* = u_*$ and $u^*$ is upper 
        semicontinuous and $u_*$ is  lower semicontinuous.  Moreover, $u$ solves the equation in 
        \eqref{convex-envelope-usual-eq-1} in the viscosity sense.

        Concerning the boundary condition, we have that
        $$
            u(x) = u^* (x) = u_* (x)  = g(x), \quad x \in \partial \Omega,
        $$
        that is, the boundary datum is attained pointwise. 

        Hence $u$ is characterized 
        as being the unique solution to \eqref{convex-envelope-usual-eq-1} and then it is 
        the usual convex envelope of $g|_{\partial \Omega}$ inside $\Omega$. 
    \end{proof}

\section*{Acknowledgments}


    LMDP and JDR are partially supported by
    Agencia Nacional de Promoción de la Investigación, el Desarrollo
    Tecnológico y la Innovación PICT-2018-3183, and
    PICT-2019-00985 and UBACYT 20020190100367 (Argentina).
    
    AQ was partially supported by Fondecyt Grant No. 1231585 (Chile).


\begin{thebibliography}{99}
                
				
		
		
		
	
	\bibitem{BarChasImb}
		G. Barles, E. Chasseigne and C. Imbert. 
		{\it On the Dirichlet problem for second-order elliptic integro-differential equations}. 
		Indiana Univ. Math. J. 57 (2008), no. 1, 213--246.
		
		   \bibitem{BdpQR} B. Barrios, L. M. Del Pezzo, A. Quaas and J. D. Rossi. {\it
The evolution problem associated with the fractional first eigenvalue.} Preprint
arXiv:2301.06524.

		
	
	\bibitem{BlancRossi} P. Blanc and J. D. Rossi. 
		{\it Games for eigenvalues of the Hessian and concave/convex envelopes.}  J. Math. Pures Appl. (9) 127 (2019), 192--215.
		
		\bibitem{Bi2}  I. Birindelli, G. Galise and D. Schiera. {\it Maximum principles and related problems for a class of nonlocal extremal operators}. Ann. Mat. Pura Appl. (4), 201(5), (2022), 2371--2412.

		
\bibitem{Bi3} I. Birindelli, G. Galise and E. Topp. {\it Fractional truncated Laplacians: representation formula, fundamental solutions and applications}. NoDEA Nonlinear Differential Equations Appl., 29(3), (2022), Paper No. 26, 49 pp.

\bibitem{HolderRef} 
I. Birindelli, G. Galise and Y. Sire. {\it Nonlocal degenerate Isaacs operators: Holder
regularity.} Preprint. arXiv:2310.11111.

\bibitem{caff}L. Caffarelli, F. Charro. {\it On a fractional Monge-Ampere operator}. Ann. PDE 1(4),  (2015), 1--34.

	\bibitem{CIL} 
		M.G. Crandall,  H. Ishii and  P.L. Lions.
		{\it User's guide to viscosity solutions of second order partial differential equations}. 
		Bull. Amer. Math. Soc. 27, (1992), 1--67.
		
		\bibitem{DelpeQR} L. M. Del Pezzo, A. Quaas and J. D. Rossi. {\it Fractional convexity.}
Math. Annalen. 383, (2022), 1687--1719.

\bibitem{Hich} E. Di Nezza, G. Palatucci and E. Valdinoci. 
        {\it Hitchhiker’s guide to the fractional Sobolev spaces.}
        Bull. Sci. Math. 136(5), (2012), 521--573.

		
		\bibitem{paper} 
		N. Guillen and R. W. Schwab. 
		{\it Aleksandrov-Bakelman-Pucci Type Estimates 
		for Integro-Differential Equations}.
		Arch. Rat. Mech. Anal. 206 (2012), 111--157.

	\bibitem{HL1} 
		F.R. Harvey and H.B. Jr. Lawson. 
		{\it Dirichlet duality and the nonlinear Dirichlet problem}, Comm. Pure Appl.
		Math. 62 (2009), 396--443.

		
	\bibitem{Ober} 
		A. M. Oberman. 
		{\it The convex envelope is the solution of a nonlinear obstacle problem}. 
		Proc. Amer. Math. Soc., 135(6) (2007), 1689--1694.
	
	\bibitem{OS} 
		A. M. Oberman and L. Silvestre. 
		{\it The Dirichlet problem for the convex envelope}. 
		Trans. Amer. Math. Soc. 363 (2011), no. 11, 5871--5886.
		
	\bibitem{Stinga}	 P. R. Stinga and J. L. Torrea. {\it Extension problem and Harnack’s inequality for some fractional operators}. Comm. Partial Differential Equations 35,
(2010), 2092--2122.
		
		
		\bibitem{RoRu} J. D. Rossi and J. Ruiz-Cases. {\it The trace fractional Laplacian and the mid-range fractional Laplacian.} Preprint. arXiv:2309.11621.

	
		
	\bibitem{Vel} M. L. J. van de Vel, Theory of Convex Structures, North Holland, Amsterdam, 1993.


	
	
\end{thebibliography}
\end{document}